\documentclass[11pt,letterpaper,reqno]{amsart}

\usepackage{amssymb}
\usepackage{amsmath}
\usepackage{amsthm}
\usepackage{amsfonts}
\usepackage{graphicx}
\usepackage{doi}

\usepackage{bookmark}
\usepackage{hyperref}
\hypersetup{pdfstartview={FitH}}

\newtheorem{theorem}{Theorem}
\newtheorem{lemma}[theorem]{Lemma}

\numberwithin{equation}{section}

\renewcommand{\leq}{\leqslant}
\renewcommand{\geq}{\geqslant}

\newcommand{\R}{\mathbb{R}}

\begin{document}

\title[Polygons of unit area and sets of infinite measure]{Polygons of unit area with vertices in sets of infinite planar measure}

\author[V. Kova\v{c}]{Vjekoslav Kova\v{c}}
\author[B. Predojevi\'{c}]{Bruno Predojevi\'{c}}

\address{Department of Mathematics, Faculty of Science, University of Zagreb, Bijeni\v{c}ka cesta 30, 10000 Zagreb, Croatia}

\email{vjekovac@math.hr}
\email{bruno.predojevic@math.hr}




\begin{abstract}
Paul Erd\H{o}s and R. Daniel Mauldin asked a series of questions on certain types of polygons of area $1$, the vertices of which can be found in every planar set of infinite Lebesgue measure. We address two of these questions, one on cyclic quadrilaterals and the other on convex polygons with congruent sides, with respectively positive and negative answers.
\end{abstract}

\maketitle


\section{Introduction}
A long time ago Paul Erd\H{o}s observed that every measurable subset of the plane of infinite two-dimensional Lebesgue measure necessarily contains vertices of a triangle of area precisely equal to $1$. He published this observation as an exercise in the Hungarian high-school journal \emph{Matematikai Lapok} and also commented on it in \cite[p.~122]{Erd78} and at several later occasions. A very short proof of this claim, based on the classical theorem of Steinhaus \cite{Ste20}, can be found in the book \cite[p.~182]{CFG91} (under Problem G13), but it can also be read between the lines of the present paper.

In the proceedings \cite{Erd83:open} of the conference \emph{Measure Theory}, held at Oberwolfach in 1983, Erd\H{o}s reported that he and R. Daniel Mauldin attempted numerous possible generalizations of the above claim. He began a list of questions by replacing triangles with quadrilaterals of special type:
\begin{quote}
\emph{Let $S$ have infinite planar measure, consider all sets of $4$ points, $x_1,x_2,x_3,x_4$, so that the area of the convex hull is $1$. Can one find $4$ such points in $S$ if we insist that they have some regularity conditions?} \cite[p.~323--324]{Erd83:open}
\end{quote}
The more specific remarks and questions followed. On the one hand, Erd\H{o}s mentioned that the answer is negative if we require that the four points, $x_1,x_2,x_3,x_4$, span a parallelogram of area $1$. He provided no proof, but one of the present authors found a very simple counterexample in \cite[\S 6]{Kov23color}. On the other hand, Erd\H{o}s claimed that the property is easily seen to hold for trapezoids of area $1$. This really is the case, as the proof can be a minor variant of the well-known solution of the above problem on triangles. However, the case of isosceles trapezoids of area $1$ was left open and it seems to be unresolved to date.

The following subsequent question attracted our attention:
\begin{quote}
\emph{(\ldots) can we assume that $(x_1,x_2,x_3,x_4)$ is inscribed in a circle?} \cite[p.~324]{Erd83:open}
\end{quote}
In this note we use only basic ideas of geometric measure theory and multivariate calculus to give a positive answer to this problem.

\begin{theorem}\label{thm:cyclic}
Every measurable planar set $\mathcal{S}$ of infinite Lebesgue measure contains the four vertices of a cyclic quadrilateral of area $1$.
\end{theorem}

It is understood that the quadrilateral needs to be non-degenerate, i.e., all $4$ of its vertices are different. Erd\H{o}s commented further:
\begin{quote}
\emph{In fact, can we find $4$ points on a circle so that the quadrilateral determined by the $4$ points should have arbitrarily large area?} \cite[p.~324]{Erd83:open}
\end{quote}
This is clearly also true, as our result immediately implies an even stronger property of every infinite measure subset $\mathcal{S}$ of the Euclidean plane: for every $\mathfrak{a}>0$ it contains vertices of a cyclic quadrilateral of area $\mathfrak{a}$. Namely, in order to find four concyclic points in $\mathcal{S}$ spanning a quadrilateral of area $\mathfrak{a}$, we simply apply Theorem \ref{thm:cyclic} to the scaled set $\mathfrak{a}^{-1/2}\mathcal{S}$.

Then Erd\H{o}s proceeded with discussing a couple of variants of the initial triangle problem, namely for isosceles or right-angled triangles of area $1$, which both seem to have remained open and we do not discuss them here. The second main result of this note is motivated by the final question in the same series, i.e., the same relevant paragraph in \cite{Erd83:open}:
\begin{quote}
\emph{Is it true that, for $n$ large enough, we have a convex polygon $(x_1,x_2,\ldots,x_n)$ of area $1$, [such that] $x_i\in S$ and all sides $(x_i,x_{i+1})$ are equal?} \cite[p.~324]{Erd83:open}
\end{quote}
Here it is not clear whether the number of vertices/sides $n$ is meant to be a fixed large integer, or it could depend on the set $S$. However, we can disprove both variants of the question by the following even stronger result.
\begin{theorem}\label{thm:congruent}
There exists a planar set $\mathcal{S}$ of infinite Lebesgue measure such that every convex polygon with congruent sides and all vertices in $\mathcal{S}$ has area strictly less than $1$.
\end{theorem}

In particular, Theorem \ref{thm:congruent} gives an infinite measure set $\mathcal{S}$ that does not contain the vertices of any convex polygon with equal side-lengths and area $1$.
Moreover, the number $1$ is not anyhow special here, but one has to be aware that all areas that are sufficiently small depending on the set $\mathcal{S}$ can certainly be obtained, simply by the Lebesgue density theorem. 

Polygons of area $1$ have also attracted some attention in the context of the so-called Euclidean Ramsey theory, where vertices of such polygons need to be found in a single color-class of an arbitrary finite coloring of the plane \cite{Gra80,DJ10,AR12,AC14}, or in a single measurable set of positive upper density \cite{Kov23color}.
However, sets of merely infinite measure, which are studied here, can be very sparse and the techniques of harmonic analysis employed in \cite{Kov23color} are no longer effective. For this reason, we approach Theorem \ref{thm:cyclic} using ``soft'' (i.e., qualitative) measure-theoretic techniques, which give very little control on the shape of the desired cyclic quadrilateral. Also, the set $\mathcal{S}$ constructed in the proof of Theorem \ref{thm:congruent} is simple enough that purely geometric observations are sufficient.


\subsection{Progress update}
After a preprint of this paper became publicly available, Junnosuke Koizumi \cite{Koi25}, ``inspired by the proof of Theorem \ref{thm:cyclic},'' showed that every subset of $\mathbb{R}^2$ of infinite measure also contains the vertices of an area $1$ isosceles triangle, right triangle, and isosceles trapezoid.
This, together with Theorems \ref{thm:cyclic} and \ref{thm:congruent}, completely resolved all five Erd\H{o}s'es questions on unit-area finite configurations in planar sets of infinite measure posed in \cite{Erd83:open}. These questions are also jointly formulated as Problem \#353 on Thomas Bloom's website \emph{Erd\H{o}s Problems} \cite{EP}.


\subsection{Notation}
Following customs of the elementary Euclidean geometry, we will write $\overline{AB}$ for the \emph{line segment} with endpoints $A$ and $B$, and $|AB|$ for the Euclidean \emph{distance} between these two points.
Next, $\mathcal{D}_r(C)$ will stand for the Euclidean open disk centered at $C$ with radius $r>0$, i.e., for the set of all points $P$ in the plane that satisfy $|CP|<r$.
When two points of $\R^2$ are given by their coordinates $(x_1,y_1)$ and $(x_2,y_2)$, their distance is rather written as $\mathop{\textup{dist}}((x_1,y_1),(x_2,y_2))$.

The two-dimensional Lebesgue measure of a measurable set $\mathcal{S}$ will be written simply as $\mathop{\textup{area}}(\mathcal{S})$, since it is a generalization of the elementary notion of \emph{area} and there is no need to make a distinction between them. 
Likewise, the one-dimensional Lebesgue measure of a measurable subset $\mathcal{I}$ of a line will be written as $\mathop{\textup{length}}(\mathcal{I})$.
The \emph{diameter} of $\mathcal{S}$ is $\sup_{A,B\in\mathcal{S}}|AB|$.
Given a set $\mathcal{I}\subseteq\R$, we write $\mathcal{I}-\mathcal{I}$ for its \emph{difference set}, defined as $\{a-b:a,b\in\mathcal{I}\}$.

The \emph{differential} of a function $f$ at a point $C$ will be written as $\textup{d}f(C)$. After we fix a coordinate system, we will start identifying linear operators from $\R^n$ to $\R^m$ with their $m\times n$ matrices in the standard bases of $\R^m$ and $\R^n$.
The \emph{operator norm} of a linear map $M\colon\R^n\to\R^m$ is defined as $\max_{|v|\leq1}|Mv|$, where $|\cdot|$ stands for the Euclidean norm (i.e., the usual vector length) on the respective spaces. 


\section{Proof of Theorem \ref{thm:cyclic}}
Suppose that $\mathcal{S}$ is a measurable subset of the plane such that $\mathop{\textup{area}}(\mathcal{S})=\infty$.
The first step of the proof is to locate a triangle $\triangle ABC$ of area $1$ with vertices in $\mathcal{S}$ in a way that we also have control over its angles and over the density of $\mathcal{S}$ at $C$. The argument is a modification of the standard one mentioned in the introduction \cite[p.~182]{CFG91}.

Consider a pair of orthogonal lines $p$ and $q$ in the plane. They subdivide the plane (up to negligible sets) into $4$ closed right-angled sectors (i.e., quadrants) $\mathcal{Q}_{1},\mathcal{Q}_{2},\mathcal{Q}_{3},\mathcal{Q}_{4}$, as in Figure \ref{fig:cyclic0}. By pigeonholing, there exists some $1\leq i\leq 4$ such that $\mathop{\textup{area}}(\mathcal{S}\cap\mathcal{Q}_{i}) =\infty$. 
Choosing the coordinate system so that $\mathcal{Q}_i$ becomes the ``upper'' sector, we can write
\[ \mathcal{Q}_i = \{ (x,y)\in\R^2 \,:\, y\geq |x| \} \]
and then the lines $p$ and $q$ respectively have Cartesian equations $y=-x$ and $y=x$; see Figure \ref{fig:cyclic1}.
For every $t\geq0$ denote the points $L_{t}=(-t,t)$ and $R_{t} = (t,t)$, so that the family of line segments $\overline{L_{t}R_{t}}$ indexed by $t\geq0$ partitions $\mathcal{Q}_{i}$. By Fubini's theorem,
\[ \mathop{\textup{area}}(\mathcal{S}\cap\mathcal{Q}_{i})
=\int_{0}^{\infty} \mathop{\textup{length}}\big(\mathcal{S}\cap\overline{L_{t}R_{t}}\big) \,\textup{d}t. \]
Therefore, there has to exist some $t>0$ such that the one-dimensional Lebesgue measure of $\mathcal{S}\cap\overline{L_{t}R_{t}}$ is strictly positive. From now on, we fix one particular such $t$ and denote $L=L_t$ and $R=R_t$. 
Define 
\[ \mathcal{I} := \big\{ x\in\R \,:\, (x,t) \in \mathcal{S}\cap\overline{LR} \big\} \subseteq [-t,t]. \] 
By the Steinhaus theorem \cite{Ste20}, the fact $\mathop{\textup{length}}(\mathcal{I})>0$ implies that there exists some $\theta>0$ such that the difference set $\mathcal{I}-\mathcal{I}$ contains the interval $(-\theta,\theta)$.
Thus, for any $0<c<\theta$, we can find two points in $\mathcal{S}\cap\overline{LR}$ at distance precisely $c$ apart from each other.

\begin{figure}
\includegraphics[width=0.3\linewidth]{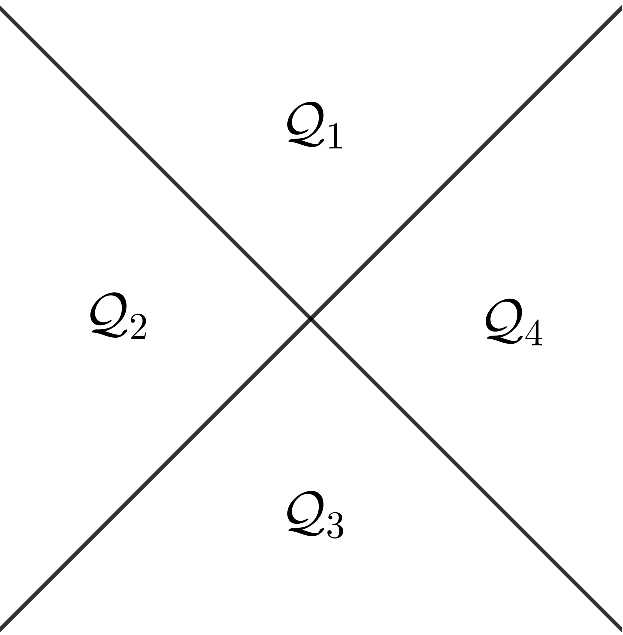}
\caption{}
\label{fig:cyclic0}
\end{figure}

Let $L'$ on $p$ and $R'$ on $q$ be the points such that $\angle L'RL = 30^\circ$ and $\angle RLR' = 30^\circ$. 
Next, let $L''$ on $p$ and $R''$ on $q$ be such that $L''R''$ is parallel to $L'R'$ and lies above (or coincides with) line $L'R'$, and that the distance between the lines $LR$ and $L''R''$ is at least $2/\theta$.
Finally, let $\mathcal{R}$ be the unbounded open region the boundary of which consists of $p$, $q$, and $L''R''$; see Figure \ref{fig:cyclic1} again.
Note that we still have $\mathop{\textup{area}}(\mathcal{S}\cap\mathcal{R})=\infty$, but from now on we only need $\mathop{\textup{area}}(\mathcal{S}\cap\mathcal{R})>0$.
Almost every point of $\mathcal{S}\cap\mathcal{R}$ is its Lebesgue density point (see e.g.\@ \cite[p.~12]{SW71} or \cite[Theorem 3.21]{Fol99}), so there exists a point $C$ in that set such that
\[ \lim_{ \varepsilon \to 0^+ } \frac{\mathop{\textup{area}}(\mathcal{S} \cap \mathcal{D}_{\varepsilon}(C))}{\mathop{\textup{area}}(\mathcal{D}_{\varepsilon}(C))} = 1. \]
In particular, there exists $\varepsilon_1>0$ such that
\begin{equation}\label{eq:LebesguePoint}
\frac{\mathop{\textup{area}}(\mathcal{S} \cap \mathcal{D}_{\varepsilon}(C))}{\mathop{\textup{area}}(\mathcal{D}_{\varepsilon}(C))} > \frac{9999}{10000} \quad\text{for every } \varepsilon\in(0,\varepsilon_1]. 
\end{equation}
Now let $c\in(0,\theta)$ be such that the distance from the point $C$ to the line $LR$ is precisely $2/c$. From the conclusion of the previous paragraph we can find points $A,B\in \mathcal{S}\cap\overline{LR}$ such that $|AB|=c$. The area of $\triangle ABC$ is then clearly $1$, while its angles satisfy
\[ \angle CBA \geq \angle CRL > \angle L'RL = 30^\circ \]
and
\[ \angle CAB \geq \angle CLR > \angle R'LR = 30^\circ. \]

\begin{figure}
\includegraphics[width=0.73\linewidth]{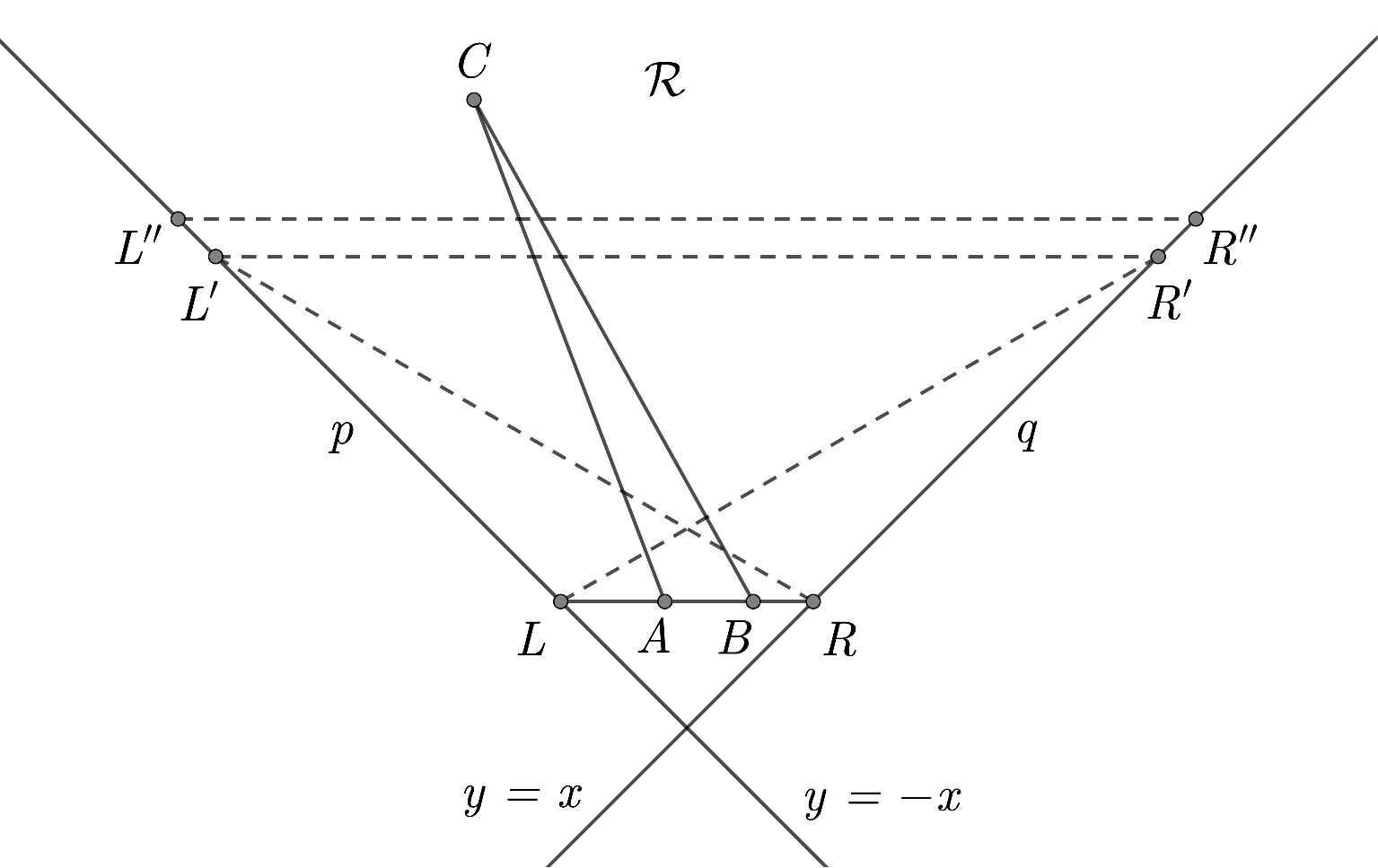}
\caption{}
\label{fig:cyclic1}
\end{figure}

We regard the positively oriented triangle $\triangle ABC$ as fixed and find it advantageous to change the coordinate system, by placing the origin at $A$ and putting $B$ on the positive part of the $x$-axis.
The new coordinates of the vertices of $\triangle ABC$ become
\[ A=(0,0), \ B=(c,0), \ C=(x_{C},y_{C}) \]
for some $x_C\in\R$ and $y_C>0$.
Lengths of its sides will be written
\[ a=|BC|, \ b=|CA|, \ c=|AB| \] 
and its angles will be denoted standardly as 
\[ \alpha=\angle BAC, \ \beta=\angle CBA, \ \gamma=\angle ACB; \]
see Figure \ref{fig:cyclic3}.
Recall that we have shown
\begin{equation}\label{eq:alphabeta}
30^\circ < \alpha,\beta < 150^\circ.
\end{equation}

The proof proceeds by ``perturbing'' the point $C$ slightly to locate $D,E$ in its small neighborhood such that $ABED$ is a non-degenerate cyclic quadrilateral of area $1$. To achieve this, we need to study analytical properties of the assignment $D\mapsto E$, which finds a unique appropriate point $E$ for many initial choices of $D$, and we only care about the cases when $D$ and $E$ are sufficiently close to $C$. Since the set $\mathcal{S}$ is ``very dense'' around $C$, it will not be able to miss all such pairs $(D,E)$.

\begin{lemma}\label{lm:thetriangle}
There exist a number $\varrho>0$, an open neighborhood $\mathcal{V}$ of the point $C$, and a $\textup{C}^1$-diffeomorphism $f\colon \mathcal{D}_{\varrho}(C) \to \mathcal{V}$ such that the following holds.
\begin{itemize}
\item One has 
\begin{equation}\label{eq:Cisfixed}
f(C)=C.
\end{equation}
\item If we denote
\begin{align*} 
\mathcal{D}^{-}_{\varrho}(C) & := \{(x,y)\in\mathcal{D}_{\varrho}(C) \,:\, y<y_C\}, \\
\mathcal{V}^{-} & := \{(x,y)\in\mathcal{V} \,:\, y<y_C\},
\end{align*}
then 
\begin{equation}\label{eq:lowerisfixed}
f(\mathcal{D}^{-}_{\varrho}(C))=\mathcal{V}^{-}.
\end{equation}
\item For every $D\in\mathcal{D}^{-}_{\varrho}(C)$ and $E\in\mathcal{V}^{-}$ one has
\begin{equation}\label{eq:whatwewant}
f(D) = E \ \Longleftrightarrow \ ABED \text{ is a cyclic quadrilateral of area } 1.
\end{equation}
\item The differential of $f$ at $C$ equals
\begin{equation}\label{eq:formulafordf}
\textup{d}f(C) = \begin{bmatrix}
1 & \frac{(\cos 2\alpha \sin\beta - \cos\alpha \sin\gamma)\sin\gamma}{\sin^3\alpha} \\[2mm]
0 & \frac{\sin^2\beta}{\sin^2\alpha}
\end{bmatrix}.
\end{equation}
\end{itemize}
\end{lemma}

\begin{figure}
\includegraphics[width=0.45\linewidth]{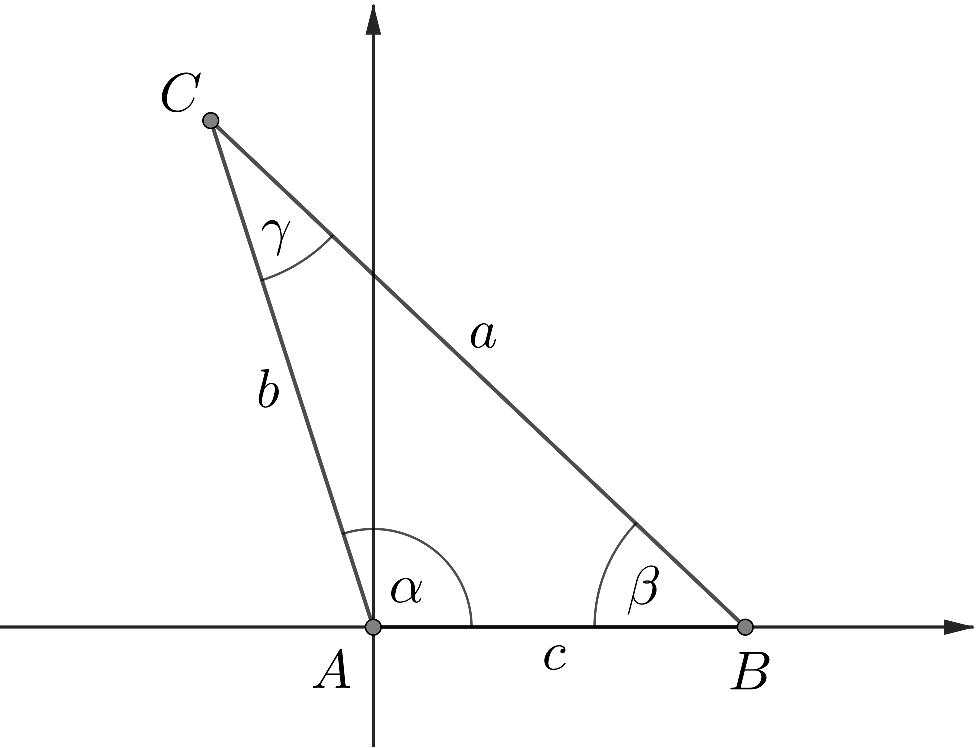}
\caption{}
\label{fig:cyclic3}
\end{figure}

The reason why we restrict ourselves to the half-disk $\mathcal{D}^{-}_{\varrho}(C)$ is that $D,A,B$ cannot be consecutive vertices of an area $1$ cyclic quadrilateral if $D$ lies on or above the line $r$ passing through $C$ and parallel to the $x$-axis; see Figure \ref{fig:cyclic2}.

\begin{proof}[Proof of Lemma \ref{lm:thetriangle}]
We will use the implicit function theorem to obtain the existence of the desired function $f\colon D\mapsto E$, defined locally, in a neighborhood of $C$, from a system of $2$ equations in $2$ variables, which encode the geometric requirements, i.e., that the quadrilateral $ABED$ is cyclic and has area $1$. We regard $D$ as fixed and try to locate the point $E=f(D)$.
Introduce the Cartesian coordinates of the two points, 
\[ D=(x_{D},y_{D}), \ E=(x_{E},y_{E}). \]

Using the so-called ``shoelace formula'' for the area in the Cartesian plane, we easily obtain
\[ \frac{1}{2} \big( 0 \cdot 0 - c \cdot 0 + c \cdot y_E - x_E \cdot 0 + x_E \cdot y_D - x_D \cdot y_E + x_D \cdot 0 - 0 \cdot y_D \big) = 1, \]
i.e., 
\begin{equation}\label{eq:Shoelace}
y_{D}x_{E} + (c-x_{D})y_{E} - 2 = 0
\end{equation}
holds whenever $ABED$ is a quadrilateral of area $1$. We conclude that $E$ needs to lie on the line $l$ given by the equation 
\begin{equation}\label{eq:equofl}
y_D x + (c-x_D) y - 2 =0;
\end{equation}
see Figure \ref{fig:cyclic2}.

Furthermore, we require that the points $A,B,E,D$ are concyclic. In order to find the equation of the circumcircle $k$ of the triangle $\triangle ABD$, we observe that its center $O(x_O,y_O)$ lies on the perpendicular bisectors of the sides $\overline{AB}$ and $\overline{AD}$, which are respectively given by
\[ x=\frac{c}{2}, \quad y=-\frac{x_{D}}{y_{D}}\Big(x-\frac{x_{D}}{2}\Big) + \frac{y_{D}}{2}. \]
Solving this system gives
\[ (x_O,y_O) = \Big( \frac{c}{2}, \frac{x_D^2 + y_D^2 - cx_D}{2y_D} \Big), \]
while the equation of circle $k$ reads
\[ (x-x_O)^2 + (y-y_O)^2 = x_O^2 + y_O^2. \]
Plugging $(x,y)=(x_{E},y_{E})$ into the last equation and simplifying, we obtain
\begin{equation}\label{eq:Circumcircle}
x_{E}^2 + y_{E}^2 - cx_{E} + \frac{c x_{D} - x_{D}^2 - y_{D}^2}{y_{D}}y_{E} = 0.
\end{equation}
The point $E$ lies at the intersection of $l$ and $k$ (see Figure \ref{fig:cyclic2} again), so its coordinates $x_E,y_E$ need to satisfy both equations \eqref{eq:Shoelace} and \eqref{eq:Circumcircle}.

\begin{figure}
\includegraphics[width=0.45\linewidth]{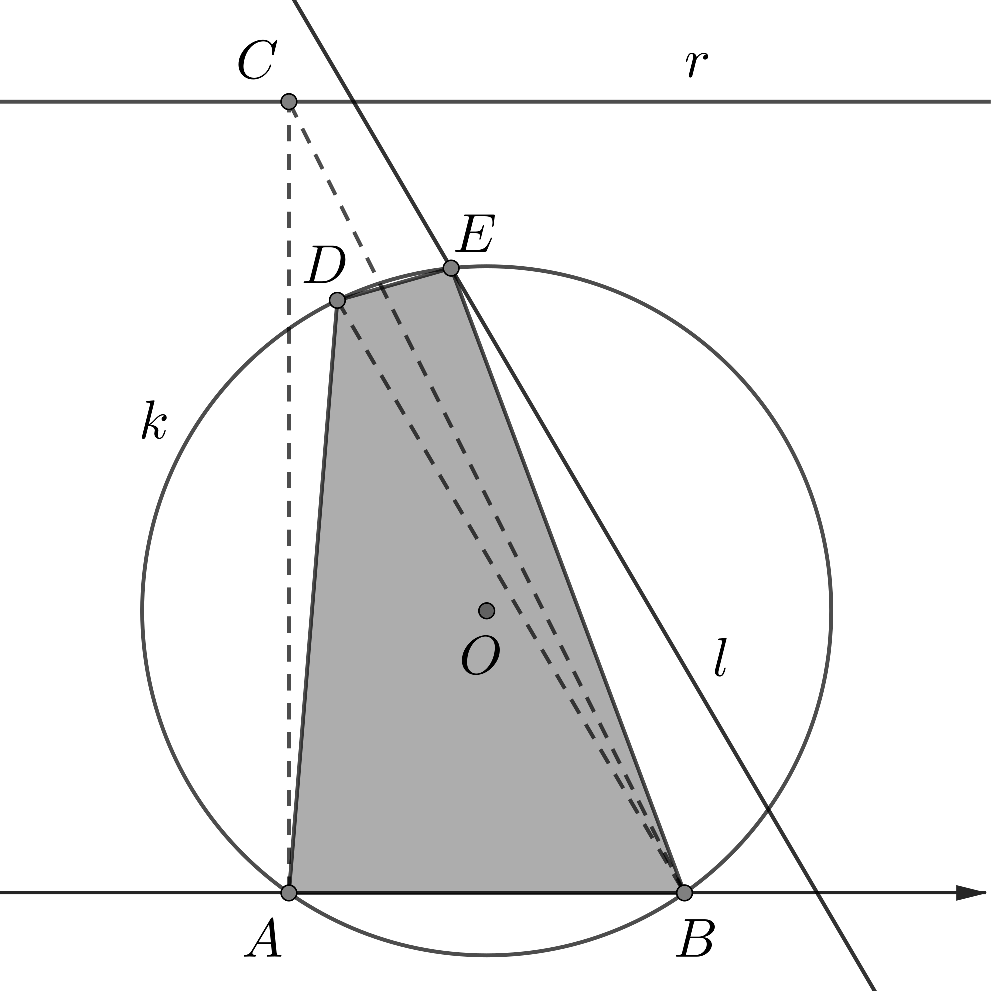}
\caption{}
\label{fig:cyclic2}
\end{figure}

Motivated by \eqref{eq:Shoelace} and \eqref{eq:Circumcircle}, we consider the function 
\[ \Phi\colon\R\times(0,\infty)\times\mathbb{R}^2 \to \mathbb{R}^2, \]  
defined by
\[ \Phi(x_1,y_1,x_2,y_2) = \Big( y_1 x_2 + (c-x_1) y_2 - 2, \ x_2^2 + y_2^2 - cx_2 + \frac{c x_1 - x_1^2 - y_1^2}{y_1}y_2 \Big). \]
Note that $\mathop{\textup{area}}(\triangle ABC)  = 1$ can be rewritten as $cy_C=2$, which easily implies 
\begin{equation}\label{eq:PhiatCC}
\Phi(C,C) = \Phi(x_{C},y_{C},x_{C},y_{C}) = (cy_C-2, 0) = (0,0).
\end{equation} 
Furthermore, we calculate the differential in the first two variables $(x_1,y_1)$ at $(C,C)$ as
\[ \textup{d}_{1}\Phi(C,C) 
=\begin{bmatrix}
-y_{C} & x_{C} \\
c-2x_{C} & \frac{x_{C}^2 - cx_{C} - y_{C}^2}{y_{C}}
\end{bmatrix} \]
and the differential in the last two variables $(x_2,y_2)$ at $(C,C)$ as
\[ \textup{d}_{2}\Phi(C,C) 
=\begin{bmatrix}
y_{C} & c-x_{C} \\
2x_{C} - c & \frac{- x_{C}^2 + cx_{C} + y_{C}^2}{y_{C}}
\end{bmatrix}. \]
Note that the determinant of the latter operator,
\[ \det \textup{d}_{2}\Phi(C,C) = (x_{C}-c)^2 + y_{C}^2, \]
is strictly positive because of $y_{C}>0$; in particular it is non-zero.
Applying the implicit function theorem we obtain open neighborhoods  $\mathcal{U}',\mathcal{V}' \subseteq \mathcal{R}$ of the point $C$ such that there exists a unique function $f\colon\mathcal{U}'\to\mathcal{V}'$ satisfying
\begin{equation}\label{eq:implexpl} 
\Phi(D,E) = 0 \ \Longleftrightarrow\ E=f(D)
\end{equation}
for every $D\in\mathcal{U}'$ and $E\in\mathcal{V}'$.
In particular, a consequence of \eqref{eq:PhiatCC} is $f(C)=C$.
Furthermore, the implicit function theorem also guarantees that $f$ is continuously differentiable and its differential can be computed by the formula
\[ \textup{d}f(C) = - \big(\textup{d}_{2}\Phi(C,C)\big)^{-1} \textup{d}_{1}\Phi(C,C). \]
After some calculation using the above expressions for $\textup{d}_{1}\Phi(C,C)$ and $\textup{d}_{2}\Phi(C,C)$, we obtain
\[ \textup{d}f(C) = -\frac{1}{(x_{C}-c)^2 + y_{C}^2}\begin{bmatrix}
\frac{-x_{C}^2+cx_{C}+y_{C}^2}{y_{C}} & x_{C}-c \\
c-2x_{C} & y_{C}
\end{bmatrix} 
\begin{bmatrix}
-y_{C} & x_{C} \\
c-2x_{C} & \frac{x_{C}^2 - cx_{C} - y_{C}^2}{y_{C}}
\end{bmatrix}, \]
i.e.,
\[ \textup{d}f(C)
= \begin{bmatrix}
1 & \frac{c(x_{C}^2 - cx_{C} - y_{C}^2)}{((x_C-c)^2+y_C^2) y_{C}} \\[2mm]
0 & \frac{x_C^2+y_C^2}{(x_C-c)^2+y_C^2}
\end{bmatrix}. \]
Using 
\[ x_C^2+y_C^2 = b^2, \quad (x_C-c)^2+y_C^2 = a^2, \]
which is clear from Figure \ref{fig:cyclic3}, we simplify $\textup{d}f(C)$ as
\[ \textup{d}f(C)
= \begin{bmatrix}
1 & \frac{c(x_{C}^2 - cx_{C} - y_{C}^2)}{a^2 y_{C}} \\[2mm]
0 & \frac{b^2}{a^2}
\end{bmatrix} \]
and then the sine theorem for $\triangle ABC$ and some trigonometric manipulation easily yield the desired formula \eqref{eq:formulafordf}.
 
Since $\det \textup{d}f(C) \neq 0$, we can apply the inverse function theorem to obtain open neighborhoods $\mathcal{U}\subseteq\mathcal{U}'$ and $\mathcal{V}\subseteq\mathcal{V}'$ of the point $C$ such that $f\colon\mathcal{U} \to \mathcal{V}$ is a $\textup{C}^1$-diffeomorphism.
By further shrinking, we can assume that $\mathcal{U}$ is an open disk $\mathcal{D}_{\varrho}(C)$ of some radius $\varrho>0$.
Besides the notation introduced in the lemma statement, also denote
\begin{align*} 
\mathcal{D}^{+}_{\varrho}(C) & := \{(x,y)\in\mathcal{D}_{\varrho}(C) \,:\, y>y_C\}, \\
\mathcal{D}^{0}_{\varrho}(C) & := \{(x,y)\in\mathcal{D}_{\varrho}(C) \,:\, y=y_C\}, \\
\mathcal{V}^{+} & := \{(x,y)\in\mathcal{V} \,:\, y>y_C\}, \\
\mathcal{V}^{0} & := \{(x,y)\in\mathcal{V} \,:\, y=y_C\}.
\end{align*}

For every $D\in\mathcal{D}^{-}_{\varrho}(C)$, the first part of the proof and \eqref{eq:implexpl} guarantee that $E=f(D)$ is the unique point lying both on the circle $k$ through $A,B,D$ and the line $l$ defined by \eqref{eq:equofl}. Since $D$ is below the line $r$, which means $\mathop{\textup{area}}(\triangle ABD)<1$, the points on the circle $k$ really appear in the order $A,B,E,D$ and form a cyclic quadrilateral of area $1$.
Moreover, $\mathop{\textup{area}}(\triangle ABE)<\mathop{\textup{area}}(\triangle ABED)$, so $E$ is also below the line $r$, i.e., $E\in\mathcal{D}^{-}_{\varrho}(C)$. This proves $f(\mathcal{D}^{-}_{\varrho}(C))\subseteq\mathcal{V}^{-}$ and tentatively also \eqref{eq:whatwewant} provided that we can show the equality in this last set inclusion, i.e., that \eqref{eq:lowerisfixed} holds.

Next, we claim that for every $D\in\mathcal{D}_{\varrho}(C)$ one has
\begin{equation}\label{eq:whenwehave0} 
D \in r \ \Longleftrightarrow\ E=f(D) \in r.
\end{equation}
In fact, \eqref{eq:whenwehave0} then further implies $E=D$, but we do not need that in the following text.
Even if this claim is geometrically evident, we give its rigorous analytical proof.
Namely, the definition of $f$ guarantees $\Phi(x_D,y_D,x_E,y_E)=(0,0)$ provided that $E=f(D)$, and the equality of the first coordinates yields
\begin{equation}\label{eq:howweget0}
y_D x_E + (c-x_D) y_E = 2 = c y_C,
\end{equation}
which enables us to express $x_E$ in terms of $x_D$, $y_D$, and $y_E$.
Now, if $y_D=y_C$, then the second equality simplifies, after some calculation, as
\[ (y_E - y_C) \underbrace{y_E \big( (x_D-c)^2 + y_C^2 \big)}_{>0} = 0, \]
which implies $y_E=y_C$.
Conversely, if $y_E=y_C$, then the second inequality easily reduces to
\[ (y_D-y_C) \underbrace{y_C (x_D^2+y_D^2)}_{>0} = 0, \]
which gives $y_D=y_C$.
(After we conclude $y_C=y_D=y_E$, then \eqref{eq:howweget0} also implies $x_E=x_D$, so that we indeed have $E=D$.)
Note that \eqref{eq:whenwehave0} can be equivalently written as 
$f(\mathcal{D}^{0}_{\varrho}(C))=\mathcal{V}^{0}$.

Finally, the continuity of $f$ and the connectedness of $\mathcal{D}^{+}_{\varrho}(C)$ imply $f(\mathcal{D}^{+}_{\varrho}(C))\subseteq\mathcal{V}^{+}$.
From all this we conclude that the image $f(\mathcal{D}^{-}_{\varrho}(C))$ needs to be precisely equal to $\mathcal{V}^{-}$.
\end{proof}

Estimating the operator norm of $\textup{d}f(C)$ by the $\ell^2$-norm of its matrix entries (the so-called Frobenius matrix norm), from \eqref{eq:formulafordf} and \eqref{eq:alphabeta} we get
\begin{align*} 
\|\textup{d}f(C)\|_{\textup{op}} & \stackrel{\eqref{eq:formulafordf}}{\leq} \sqrt{1^2 + \bigg(\frac{(\cos 2\alpha \sin\beta - \cos\alpha \sin\gamma)\sin\gamma}{\sin^3\alpha}\bigg)^2 + \Big(\frac{\sin^2\beta}{\sin^2\alpha}\Big)^2} \\
& \stackrel{\eqref{eq:alphabeta}}{\leq} \sqrt{1 + \Big(\frac{2}{\sin^3 30^\circ}\Big)^2 + \Big(\frac{1}{\sin^2 30^\circ}\Big)^2} < 20. 
\end{align*}
Also,
\[ \det(\textup{d}f(C)) \stackrel{\eqref{eq:formulafordf}}{=} \frac{\sin^2\beta}{\sin^2\alpha} \stackrel{\eqref{eq:alphabeta}}{\geq} \sin^2 30^\circ = \frac{1}{4}. \]
By continuity of the maps
\[ P \mapsto \|\textup{d}f(P)\|_{\textup{op}} \quad\text{and}\quad P \mapsto \det(\textup{d}f(P)) \]
there exists $0<\varepsilon_2\leq\varrho$ such that
\begin{equation}\label{eq:theestimate1}
\|\textup{d}f(P)\|_{\textup{op}} <20 \quad\text{for every } P\in\mathcal{D}_{\varepsilon_2}(C)
\end{equation}
and
\begin{equation}\label{eq:theestimate2}
\det(\textup{d}f(P))>\frac{1}{5} \quad\text{for every } P\in\mathcal{D}_{\varepsilon_2}(C).
\end{equation}
Next, by the mean value inequality for vector-valued functions, we conclude that
\[ |Cf(P)| \stackrel{\eqref{eq:Cisfixed}}{=} |f(C)f(P)| \stackrel{\eqref{eq:theestimate1}}{\leq} 20 |CP| \]
holds for every $P\in\mathcal{D}_{\varepsilon_2}(C)$, so
\begin{equation}\label{eq:theinclusion}
f(\mathcal{D}_{\varepsilon}(C)) \subseteq \mathcal{D}_{20\varepsilon}(C) \quad\text{for every } \varepsilon\in(0,\varepsilon_2]. 
\end{equation}
Let $\mathcal{D}_{\varepsilon_3}(C)$, for some $\varepsilon_3>0$, be a disk around $C$ fully contained in $\mathcal{V}$. 

\begin{figure}
\includegraphics[width=0.75\linewidth]{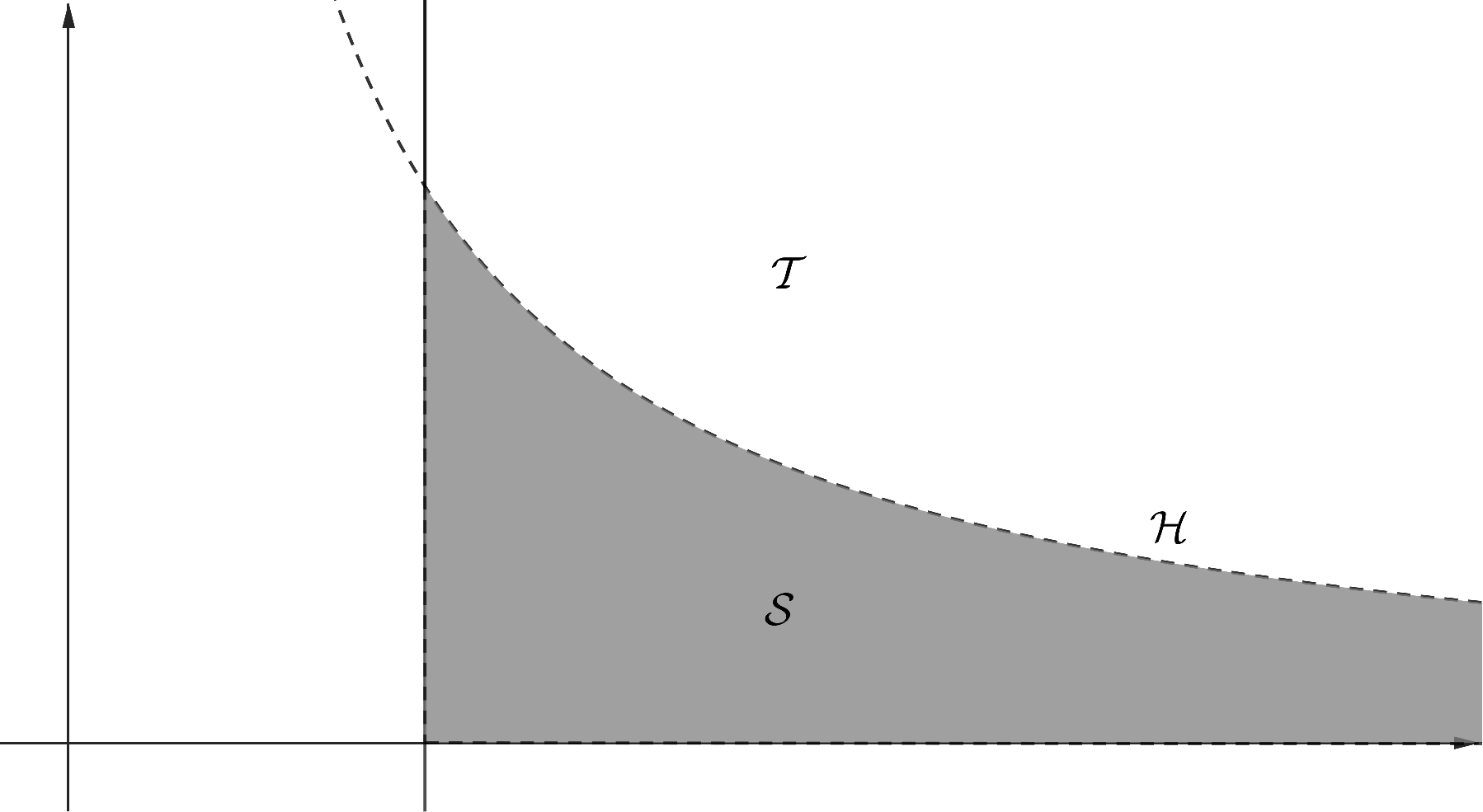}
\caption{}
\label{fig:conv1}
\end{figure}

From \eqref{eq:LebesguePoint} we conclude
\[ \frac{1}{2}\cdot\frac{\mathop{\textup{area}}(\mathcal{S} \cap \mathcal{D}^{-}_{\varepsilon}(C))}{\mathop{\textup{area}}(\mathcal{D}^{-}_{\varepsilon}(C))}  + \frac{1}{2}\cdot\underbrace{\frac{\mathop{\textup{area}}(\mathcal{S} \cap \mathcal{D}^{+}_{\varepsilon}(C))}{\mathop{\textup{area}}(\mathcal{D}^{+}_{\varepsilon}(C))}}_{\leq1} \stackrel{\eqref{eq:LebesguePoint}}{>} \frac{9999}{10000} \]
whenever $\varepsilon\leq\varepsilon_1$, so that
\begin{equation}\label{eq:LebesguePoint2}
\frac{\mathop{\textup{area}}(\mathcal{S} \cap \mathcal{D}^{-}_{\varepsilon}(C))}{\mathop{\textup{area}}(\mathcal{D}^{-}_{\varepsilon}(C))} > \frac{4999}{5000} \quad\text{for every } \varepsilon\in(0,\varepsilon_1]. 
\end{equation}
Finally, take
\[ \delta := \min\{\varepsilon_1,\varepsilon_2,\varepsilon_3\} \]
and define the set
\[ \mathcal{A} := \mathcal{S} \cap \mathcal{D}^{-}_{\delta}(C), \]
so that \eqref{eq:LebesguePoint2} immediately gives
\[ \mathop{\textup{area}}(\mathcal{A}) > \frac{4999}{5000} \mathop{\textup{area}}(\mathcal{D}^{-}_{\delta}(C)). \]
Also define
\[ \mathcal{B} := f\big(\mathcal{S} \cap \mathcal{D}^{-}_{\delta/20}(C)\big), \]
so that \eqref{eq:lowerisfixed} and \eqref{eq:theinclusion} guarantee $\mathcal{B}\subseteq\mathcal{D}^{-}_{\delta}(C)$, while the change of variables formula (e.g. \cite[Theorem 2.47]{Fol99}), the fact that $f$ is $\textup{C}^1$-diffeomorphism, \eqref{eq:theestimate2}, and \eqref{eq:LebesguePoint2} give
\begin{align*}
\mathop{\textup{area}}(\mathcal{B})
& = \int_{\mathcal{S} \cap \mathcal{D}^{-}_{\delta/20}(C)} |\det(\textup{d}f(P))| \,\textup{d}P 
\stackrel{\eqref{eq:theestimate2}}{\geq} \frac{1}{5}\mathop{\textup{area}}\big(\mathcal{S} \cap \mathcal{D}^{-}_{\delta/20}(C)\big) \\
& \stackrel{\eqref{eq:LebesguePoint2}}{>} \frac{1}{5} \cdot\frac{4999}{5000}\cdot \frac{1}{20^2} \mathop{\textup{area}}(\mathcal{D}^{-}_{\delta}(C))
> \frac{1}{5000}\mathop{\textup{area}}(\mathcal{D}^{-}_{\delta}(C)).
\end{align*}
Therefore, $\mathcal{A}$ and $\mathcal{B}$ are two measurable subsets of $\mathcal{D}^{-}_{\delta}(C)$ satisfying
\[ \mathop{\textup{area}}(\mathcal{A}) + \mathop{\textup{area}}(\mathcal{B}) > \mathop{\textup{area}}(\mathcal{D}^{-}_{\delta}(C)), \]
so they cannot be disjoint.
Let $E$ be a point from the intersection $\mathcal{A}\cap\mathcal{B}$. By the definition of $\mathcal{A}$ we have $E\in\mathcal{S} \cap \mathcal{D}^{-}_{\varepsilon_3}(C)\subseteq \mathcal{S}\cap\mathcal{V}^{-}$, while the definition of $\mathcal{B}$ gives a point $D\in\mathcal{S}\cap\mathcal{D}^{-}_{\varrho}(C)$ such that $E=f(D)$.
Now we know that $A,B,D,E\in\mathcal{S}$ and \eqref{eq:whatwewant} finalizes the proof of Theorem \ref{thm:cyclic} by concluding that $ABED$ is a cyclic quadrilateral of unit area.


\section{Proof of Theorem \ref{thm:congruent}}

An example of a set with the desired property is
\[ \mathcal{S} := \{ (x,y)\in\R^2 \,:\, x>1,\, y>0,\, 4xy<1 \}. \]
Its upper boundary is a piece of hyperbola,
\[ \mathcal{H} := \{ (x,y)\in\R^2 \,:\, x>1,\, 4xy=1 \}, \]
which is, at the same time, the lower boundary of the closed convex set
\[ \mathcal{T} := \{ (x,y)\in\R^2 \,:\, x\geq1,\, 4xy\geq1 \}; \]
see Figure \ref{fig:conv1}.
Note that
\[ \mathop{\textup{area}}(\mathcal{S}) = \int_{1}^{\infty} \frac{\textup{d}x}{4x} = \infty. \]

\begin{figure}
\includegraphics[width=0.73\linewidth]{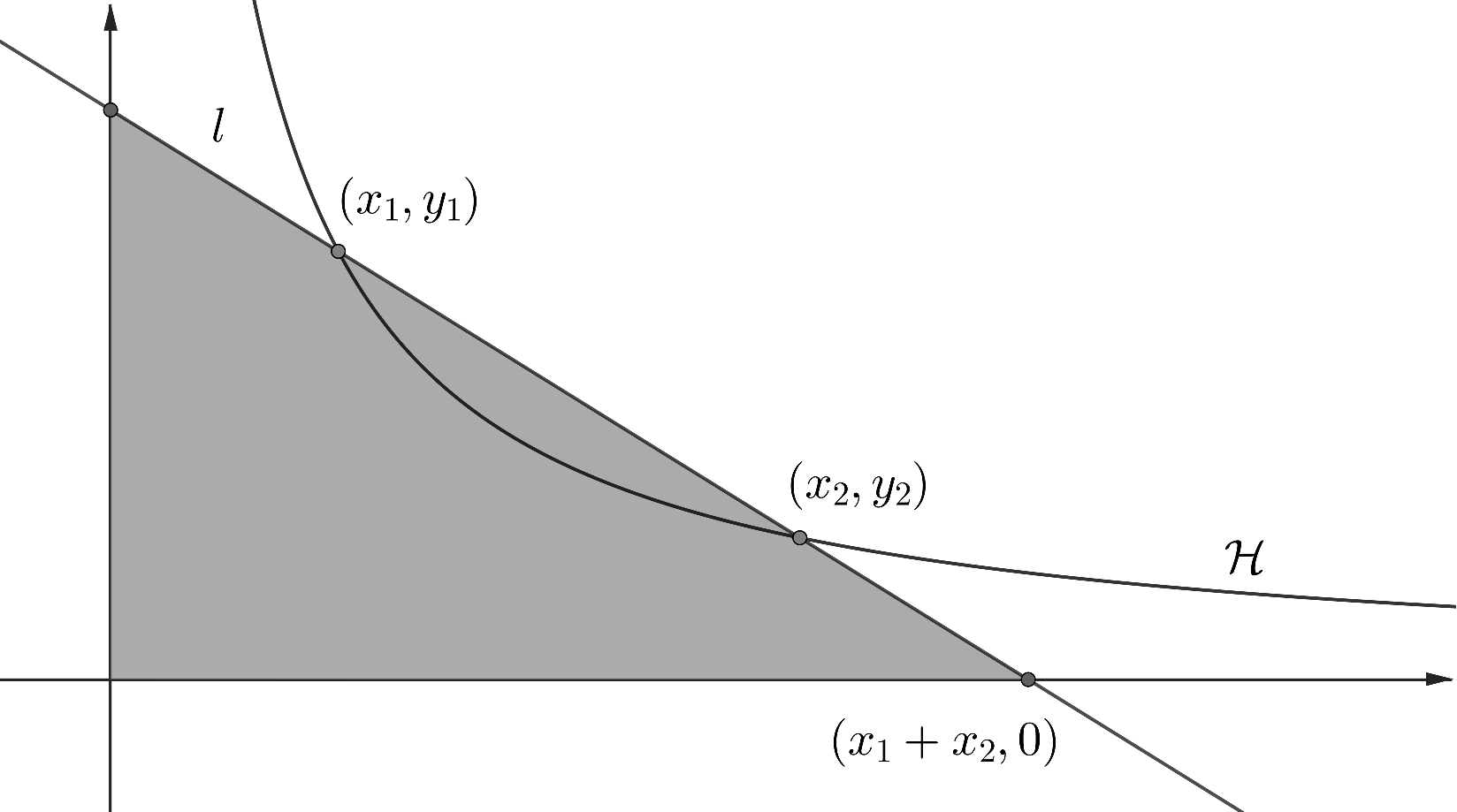}
\caption{}
\label{fig:conv2}
\end{figure}

\begin{lemma}\label{lm:hyparea}
Let $l$ be the line through any two points $(x_1,y_1)$ and $(x_2,y_2)$ on $\mathcal{H}$.
Then the line $l$ intersects the $x$-axis at the point $(x_1+x_2,0)$ and the triangle determined by $l$ and the coordinate axes has area precisely
\begin{equation}\label{eq:trianglearea}
\frac{1}{2} + \frac{(x_1-x_2)^2}{8 x_1 x_2}. 
\end{equation}
In the limiting case, when the two points on $\mathcal{H}$ coincide, the line $l$ becomes the tangent on $\mathcal{H}$ at $(x_1,y_1)$ and the above area is always equal to $1/2$.
\end{lemma}

\begin{proof}[Proof of Lemma \ref{lm:hyparea}]
Consult Figure \ref{fig:conv2}. The equation of the line through $(x_1,y_1)$ and $(x_2,y_2)$ reads
\[ y - y_1 = \frac{y_2-y_1}{x_2-x_1} (x - x_1), \]
while $y_1=1/(4x_1)$ and $y_2=1/(4x_2)$ enable its simplification as
\[ y = \frac{-x+x_1+x_2}{4x_1 x_2}. \]
The two segments that $l$ cuts of from the axes have lengths $x_1+x_2$ and $(x_1+x_2)/(4x_1 x_2)$, so the area of the triangle in question equals $(x_1+x_2)^2/(8x_1 x_2)$, which easily transforms into \eqref{eq:trianglearea}. 

A similar computation finds the equation of the tangent line on $\mathcal{H}$ at $(x_1,y_1)$ as
\[ y = \frac{-x+2x_1}{4x_1^2}, \]
so the segments cut off from the axes now have lengths $2x_1$ and $1/(2x_1)$, while the formed triangle has area $1/2$.
\end{proof}

Now we turn to the actual proof of the second theorem.
Suppose that there is a convex polygon $\mathcal{P}$ with $\mathop{\textup{area}}(\mathcal{P})\geq1$, all sides of length $a>0$, and all vertices contained in the above set $\mathcal{S}$. We will gradually narrow down several possibilities until we arrive at a final contradiction.

Let us first suppose that $\mathcal{P}$ is fully contained in the region $\mathcal{S}$; see Figure \ref{fig:conv3}. Then $\mathcal{P}$ and $\mathcal{T}$ are disjoint convex sets, and they can be separated by a line, via an elementary version of the Hahn-Banach separation theorem. By moving the separation line upwards until it touches $\mathcal{T}$, and possibly rotating it counterclockwise if the touching point happens to be $(1,1/4)$, we obtain a tangent $l$ on $\mathcal{H}$ such that the polygon $\mathcal{P}$ lies below it. The area of $\mathcal{P}$ can then be bounded from above by the area of the triangle formed by $l$ and the coordinate axes, so Lemma \ref{lm:hyparea} gives $\mathop{\textup{area}}(\mathcal{P})\leq 1/2$, which is a contradiction.

\begin{figure}
\includegraphics[width=0.6\linewidth]{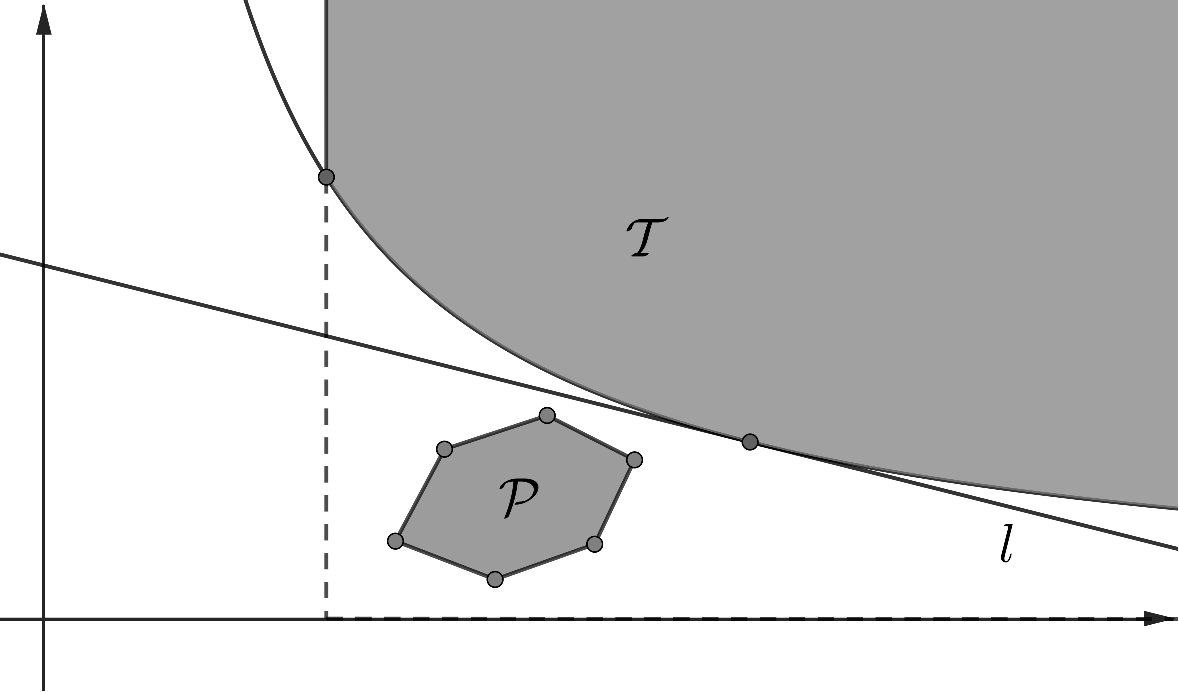}
\caption{}
\label{fig:conv3}
\end{figure}

Thus, $\mathcal{T}$ actually intersects the polygon $\mathcal{P}$, but it can intersect only one of its sides contained in the upper boundary of $\mathcal{P}$, call it $\overline{AB}$; see Figure \ref{fig:conv4}. Otherwise, it would also contain some of the vertices of $\mathcal{P}$, but we assumed that all of the vertices are in the region $\mathcal{S}$.
Let $(x_1,y_1)$ and $(x_2,y_2)$, for $x_1\leq x_2$, be the intersection points of $\mathcal{H}$ and the side $\overline{AB}$, allowing the possibility that the two points coincide.
Also, let $l$ be the line through $(x_1,y_1)$ and $(x_2,y_2)$, or the line tangent to $\mathcal{H}$ at $(x_1,y_1)=(x_2,y_2)$.
Observing that $\mathcal{P}$ is, by its convexity, again contained in the triangle formed by $l$ and the coordinate axes, the formula \eqref{eq:trianglearea} from Lemma \ref{lm:hyparea} now yields
\[ \mathop{\textup{area}}(\mathcal{P})\leq \frac{1}{2} + \frac{(x_1-x_2)^2}{8 x_1 x_2} \leq \frac{1}{2} + \frac{a^2}{8 x_1 x_2}. \]
We clearly arrive at the contradiction with $\mathop{\textup{area}}(\mathcal{P})\geq1$ again if
\begin{itemize}
\item $a<2$ \ (as $x_2\geq x_1\geq 1$),
\item or $x_1\geq 7a/8$ \ (as then also $x_2\geq7a/8$).
\end{itemize}

\begin{figure}
\includegraphics[width=0.57\linewidth]{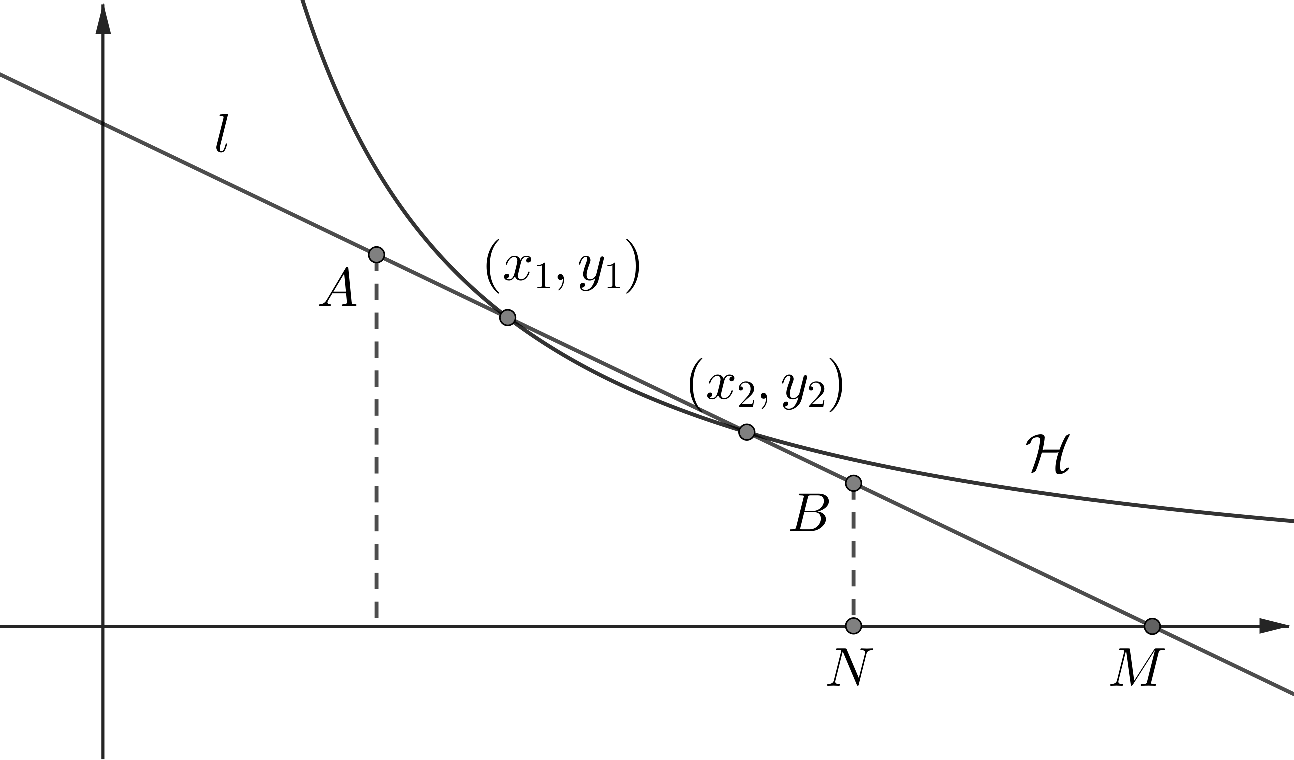}
\caption{}
\label{fig:conv4}
\end{figure}

Thus, in the following we additionally assume that $a\geq 2$ and $x_1<7a/8$.
Without loss of generality, $A$ has a smaller $x$-coordinate than $B$. However, $A$ is then contained in the rectangle $[1,7a/8]\times[0,1/4]$, the diagonal of which is
\[ \sqrt{\Big(\frac{7a}{8}-1\Big)^2+\Big(\frac{1}{4}\Big)^2}<a, \] 
and there is no room inside that rectangle for another side of the polygon $\mathcal{P}$. Consequently, $A$ is the leftmost vertex of $\mathcal{P}$.
Next, recall that all vertices of $\mathcal{P}$ except $A$ and $B$ lie below the line $l$. 
Let $M$ be the intersection of $l$ with the $x$-axis, and let $N$ be the foot of the perpendicular from $B$ onto the $x$-axis; see Figure \ref{fig:conv4} again. 
By another claim from Lemma \ref{lm:hyparea}, the coordinates of $M$ are $(x_1+x_2,0)$.
The diameter of triangle $\triangle BNM$ is
\[ |BM| < \mathop{\textup{dist}}\big( (x_2,y_2), (x_1+x_2,0) \big) \leq x_1 + \frac{1}{4} < \frac{7a}{8} + \frac{1}{4} \leq a, \]
which leaves no room inside $\triangle BNM$ for another side of $\mathcal{P}$. Thus, $B$ is the rightmost vertex of the polygon $\mathcal{P}$.

We can now denote the vertices of $\mathcal{P}$ counterclockwise as
\[ A= C_1,\, C_2,\, \ldots,\, C_{n-1},\, C_n=B, \]
knowing that $n\geq 3$; see Figure \ref{fig:conv5}. Let $C'_i$ be the projection of $C_i$ onto the $x$-axis for $i=1,\ldots,n$.
For every $1\leq i\leq n-1$ we know that the projection of $\overline{C_i C_{i+1}}$ onto the $y$-axis has length at most $1/4$, so the Pythagorean theorem and $a\geq 2$ give
\[ a^2 = |C_i C_{i+1}|^2 \leq |C'_i C'_{i+1}|^2 + \Big(\frac{1}{4}\Big)^2 \leq |C'_i C'_{i+1}|^2 + \frac{a^2}{64}, \]
which implies $|C'_i C'_{i+1}|\geq 7a/8$.
Ultimately,
\[ a = |AB| \geq |C'_1 C'_n| = \sum_{i=1}^{n-1} |C'_i C'_{i+1}| \geq \frac{7(n-1)a}{8} >a, \]
which leads to the ultimate contradiction and completes the proof of Theorem \ref{thm:congruent}.

\begin{figure}
\includegraphics[width=0.77\linewidth]{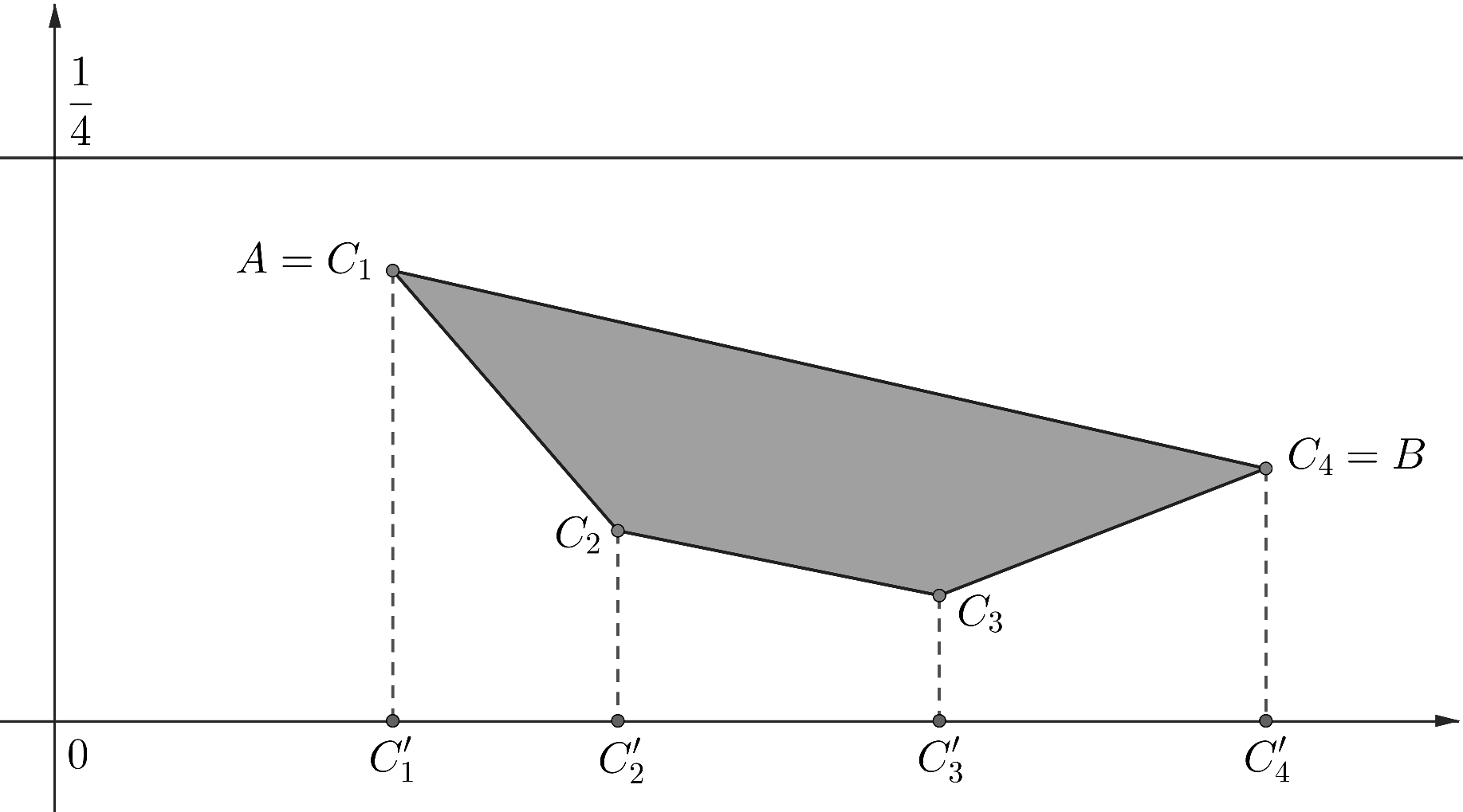}
\caption{}
\label{fig:conv5}
\end{figure}


\section*{Acknowledgment}
This work was supported in part by the Croatian Science Foundation under the project HRZZ-IP-2022-10-5116 (FANAP).



\bibliography{PolygonsofUnitArea}{}

@article {Erd78,
    AUTHOR = {Erd\H{o}s, Paul},
     TITLE = {Set-theoretic, measure-theoretic, combinatorial, and number-theoretic problems concerning point sets in {E}uclidean space},
   JOURNAL = {Real Anal. Exchange},
  FJOURNAL = {Real Analysis Exchange},
    VOLUME = {4},
      YEAR = {1978/79},
    NUMBER = {2},
     PAGES = {113--138},
      ISSN = {0147-1937},
   MRCLASS = {04A20 (05C55 10A20 52A37)},
  MRNUMBER = {533932}
}

@incollection {Erd83:open,
    AUTHOR = {Erd\H{o}s, Paul},
     TITLE = {Some combinatorial, geometric and set theoretic problems in measure theory},
 BOOKTITLE = {Measure {T}heory, {O}berwolfach 1983: {P}roceedings of the {C}onference held at {O}berwolfach, {J}une 26--{J}uly 2, 1983},
    SERIES = {Lecture Notes in Mathematics},
    VOLUME = {1089},
     PAGES = {321--327},
 PUBLISHER = {Springer, Berlin, Heidelberg},
      YEAR = {1984},
    EDITOR = {K\"{o}lzow, D. and Maharam-Stone, D.},
       DOI = {10.1007/BFb0072626}
}

@book {CFG91,
    AUTHOR = {Croft, Hallard T. and Falconer, Kenneth J. and Guy, Richard K.},
     TITLE = {Unsolved problems in geometry},
    SERIES = {Problem Books in Mathematics},
      NOTE = {Unsolved Problems in Intuitive Mathematics, II},
 PUBLISHER = {Springer-Verlag, New York},
      YEAR = {1991},
     PAGES = {xvi+198},
      ISBN = {0-387-97506-3},
   MRCLASS = {52-02},
  MRNUMBER = {1107516},
       DOI = {10.1007/978-1-4612-0963-8}
}

@misc {EP,
    author = {Bloom, Thomas F.},
    title = {{Erd\H{o}s} {P}roblem \#353},
    howpublished = {\url{https://www.erdosproblems.com/353}},
    note = {Accessed: December 16, 2024}
}

@unpublished {Kov23color,
	Author = {Kova\v{c}, Vjekoslav},
	Title = {Coloring and density theorems for configurations of a given volume (preprint, 47 pp.)},
	Note = {Available at: arXiv:2309.09973},
    Year = {2024}
}

@article {Ste20,
    AUTHOR = {Steinhaus, Hugo},
     TITLE = {Sur les distances des points dans les ensembles de mesure positive},
   JOURNAL = {Fund. Math.},
  FJOURNAL = {Fundamenta Mathematicae},
    VOLUME = {1},
      YEAR = {1920},
     PAGES = {93--104},
       DOI = {10.4064/fm-1-1-93-104}
}

@article {Gra80,
    AUTHOR = {Graham, Ronald Lewis},
     TITLE = {On partitions of {${\mathbb{E}}\sp{n}$}},
   JOURNAL = {J. Combin. Theory Ser. A},
  FJOURNAL = {Journal of Combinatorial Theory. Series A},
    VOLUME = {28},
      YEAR = {1980},
    NUMBER = {1},
     PAGES = {89--97},
      ISSN = {0097-3165,1096-0899},
   MRCLASS = {05A17 (05C55)},
  MRNUMBER = {558877},
MRREVIEWER = {Karsten\ Steffens},
       DOI = {10.1016/0097-3165(80)90061-8},
       URL = {https://doi.org/10.1016/0097-3165(80)90061-8},
}

@article {DJ10,
    AUTHOR = {Dumitrescu, Adrian and Jiang, Minghui},
     TITLE = {Monochromatic simplices of any volume},
   JOURNAL = {Discrete Math.},
  FJOURNAL = {Discrete Mathematics},
    VOLUME = {310},
      YEAR = {2010},
    NUMBER = {4},
     PAGES = {956--960},
      ISSN = {0012-365X,1872-681X},
   MRCLASS = {05D10 (52A38)},
  MRNUMBER = {2574848},
       DOI = {10.1016/j.disc.2009.09.026},
       URL = {https://doi.org/10.1016/j.disc.2009.09.026},
}

@article {AC14,
    AUTHOR = {Adhikari, Sukumar Das and Chen, Yong-Gao},
     TITLE = {On monochromatic configurations for finite colorings},
   JOURNAL = {Discrete Math.},
  FJOURNAL = {Discrete Mathematics},
    VOLUME = {333},
      YEAR = {2014},
     PAGES = {106--109},
      ISSN = {0012-365X,1872-681X},
   MRCLASS = {05C15},
  MRNUMBER = {3233412},
       DOI = {10.1016/j.disc.2014.06.015},
       URL = {https://doi.org/10.1016/j.disc.2014.06.015},
}

@article {AR12,
    AUTHOR = {Adhikari, Sukumar Das and Rath, Purusottam},
     TITLE = {Remarks on monochromatic configurations for finite colorings of the plane},
   JOURNAL = {Note Mat.},
  FJOURNAL = {Note di Matematica},
    VOLUME = {32},
      YEAR = {2012},
    NUMBER = {2},
     PAGES = {83--88},
      ISSN = {1123-2536,1590-0932},
   MRCLASS = {05D10},
  MRNUMBER = {3071796},
       DOI = {10.1109/mcs.2012.2185888},
       URL = {https://doi.org/10.1109/mcs.2012.2185888},
}

@book {Fol99,
    AUTHOR = {Folland, Gerald B.},
     TITLE = {Real analysis},
    SERIES = {Pure and Applied Mathematics (New York)},
   EDITION = {Second},
      NOTE = {Modern techniques and their applications, A Wiley-Interscience Publication},
 PUBLISHER = {John Wiley \& Sons, Inc., New York},
      YEAR = {1999},
     PAGES = {xvi+386},
      ISBN = {0-471-31716-0},
   MRCLASS = {00A05 (26-01 28-01 46-01)},
  MRNUMBER = {1681462},
}

@book {SW71,
    AUTHOR = {Stein, Elias M. and Weiss, Guido},
     TITLE = {Introduction to {F}ourier analysis on {E}uclidean spaces},
    SERIES = {Princeton Mathematical Series},
    VOLUME = {No. 32},
 PUBLISHER = {Princeton University Press, Princeton, NJ},
      YEAR = {1971},
     PAGES = {x+297},
   MRCLASS = {42A92 (31B99 32A99 46F99 47G05)},
  MRNUMBER = {304972},
MRREVIEWER = {Edwin\ Hewitt},
}

@article {Koi25,
    AUTHOR = {Koizumi, Junnosuke},
     TITLE = {Isosceles trapezoids of unit area with vertices in sets of
              infinite planar measure},
   JOURNAL = {Proc. Amer. Math. Soc.},
  FJOURNAL = {Proceedings of the American Mathematical Society},
    VOLUME = {153},
      YEAR = {2025},
    NUMBER = {11},
     PAGES = {4753--4758},
      ISSN = {0002-9939,1088-6826},
   MRCLASS = {28A75 (52C10)},
  MRNUMBER = {4971569},
       DOI = {10.1090/proc/17322},
       URL = {https://doi.org/10.1090/proc/17322},
}
\bibliographystyle{plainurl}

\end{document}